\documentclass[titlepage,draft,12pt]{article} 
\usepackage{amsfonts,latexsym} 
\usepackage{pstricks,pst-plot}
\usepackage{amsmath} 
\usepackage{amsthm}
\usepackage[a4paper]{geometry}

\date{}

\newcommand{\re}{\mathbb{R}}

\newtheorem{thm}{Theorem}[section]

\newtheorem{rmk}[thm]{Remark}

\newtheorem{cor}[thm]{Corollary}

\def\limsup{\mathop{\overline{\rm lim}}}

\def\L{\displaystyle\limsup}

\title{A sharp stability criterion for single well Duffing and Duffing-like equations}

\author{
Alain Haraux\vspace{1ex}\\ 
{\normalsize Universit\'{e} Pierre et Marie Curie} \\
{\normalsize Laboratoire Jacques-Louis Lions}\\ 
{\normalsize PARIS (France)}\\  
{\normalsize e-mail: \texttt{haraux@ann.jussieu.fr}}}

\begin{document}
\maketitle
\begin{abstract}
	We refine some previous sufficient conditions for exponential stability of the linear ODE  $$ u''+ cu' + (b+a(t))u = 0$$ where $b, c>0$ and $a$ is a bounded nonnegative time dependent coefficient. 
	This allows  to improve some results on uniqueness and asymptotic stability of periodic or almost periodic solutions of the equation$$ u''+ cu' + g(u)=
f(t) $$where $c>0$,  $f \in L^\infty (\re)$  and $g\in C^1(\re)$ satisfies some sign hypotheses. The typical case is $ g(u) = bu + a\vert u\vert^p u $ with $a\ge 0 , b>0.$ Similar properties are valid for evolution equations of the form $$ u''+ cu' + (B+A(t))u = 0$$  where $A(t) $ and $B$ are self-adjoint operators on a real Hilbert space $H$ with $B$ coercive and $A(t)$ bounded in $L(H)$ with a sufficiently small bound of  its norm in $L^{\infty}(\re^+, L(H))$ .
\\

\noindent
{\textbf Mathematics Subject Classification 2010 (MSC2010):}
34D 23, 34 F15, 34 K13, 35L 10

\vspace{1cm} 

\noindent{\textbf Key words:} Second order ODE, bounded solutions, evolution equations, exponential stability.
\end{abstract}

 
\section{Introduction}

Our starting  point is a paper by W.S. Loud concerning the second order ODE 
\begin{equation}\label{Duff} u''+ cu' + g(u)=
f(t)\end{equation} where $c>0$,  $f \in L^\infty ([t_0, +\infty))$  and $g\in C^1(\re)$
satisfies some sign hypotheses. The typical case is 
$$ g(u) = bu + a\vert u\vert^p u
$$ with $a\ge 0 , b>0$ More generally , assuming  $g(0) = 0$  and \begin{equation} \forall s\in {\re},\quad g'(s)
\ge  b
\end{equation}  for some $b>0$,  W.S. Loud \cite{Loud2}, refining some previous results from M. Cartwright and J. Littlewood \cite{C-L}, established that all solutions of
(1) are ultimately bounded and gave rather sharp estimates of their ultimate bound.  The proof of these  estimates relied on a delicate geometrical argument in the phase space. More than 50 years later, in \cite{F-H2},  F. Fitouri and the author, by a purely analytical method consisting in a suitable combination of differential inequalities, obtained the improved estimate on $u$ \begin{equation} \label {CHbound}\displaystyle\limsup_{ \,t \rightarrow +\infty }\vert u(t)\vert \le \max \{ {2\over {c\sqrt b}}, {1\over {b}}\}\
 \displaystyle\limsup_{ \,t \rightarrow +\infty }\vert f(t)\vert \end{equation} 

 In the original  paper  \cite{Loud2}, it was also shown that under a smallness condition on $||f||_{\infty}$, all solutions are asymptotic to each other at infinity and in particular if $f $ is periodic, there is only one periodic solution. On the other hand it has been  known  long ago that for large $f$, there may exist several  periodic solutions, cf. \cite{Lev, Loud3}. In \cite{F-H2}, Loud's result on uniqueness and asymptotic stability was also refined  by working on stability properties of the dissipative Hill's equation  \begin{equation} \label {Hilldiss} u''+ cu' + (b+a(t))u = 0\end{equation} where $b, c>0$ and $a$ is a bounded time dependent coefficient with $a\ge 0$. \\
 
 The stability question for Duffing equations as well as nonlinear infinite dimensional systems having a similar structure has an important impact on the theoretical study of engineering problems such as stability of bridges, cranes, structural stability of mechanical devices involving elastic beams and plates. There is a vast litterature on these questions covering many important issues of applied research, cf. e.g. \cite{FilG ?, FG, GaGa, FGa, GGH} and the references therein. \\
 
 The main point of the present paper is an improvement of the sufficient conditions on $a$ insuring the asymptotic (actually exponential) stability of the null solution to  equation \eqref{Hilldiss}.
 The plan of the paper is the following: in Section 2 we state and prove an improved condition of stability and we discuss the sharpness of that condition. In Section 3 we derive the consequences for Duffing's equation \eqref{Duff} with some comments on optimality. Sections 4 and 5 are devoted to a class of  wave equations to which one of our main results can be extended.

\setcounter{equation}{0}
\section{A simple exponential stability result}\label{sec:abstr} Equation \eqref{Hilldiss} has been the object of many studies, either without dissipation (case $c=0$) or with dissipation (case $c>0$). The case without dissipation is difficult even when the coefficient $a(t)$ is a simple trigonometric function, since complicated resonance phenomena can occur. The dissipative case can be reduced to it by a simple transformation, but then the global information on solutions is difficult to transfer back to the original problem. Here we treat directly the dissipative case by the method of Liapunov functions. 
\subsection{Main result}\label{main}

\begin{thm}\label {mainth} Let $J \subset \re$ be an open interval and assume that $a: = a(t) \in L^{\infty}(J)$ with  \begin{equation} \label {C} 0\le a(t) \le C,\quad a.e.  in J. \end{equation} Assume \begin{equation}\label{C2} C< c\max\{c, 2\sqrt b \}. \end{equation} Then there are  $\delta>0$ and $M>0$ such that any solution $u\in W^{2,\infty} $ of \eqref{Hilldiss} in $J$ satisfies  \begin{equation}u^2(t) + u'^2(t) \le M [u^2(s) + u'^2(s)]\exp(-\delta (t-s))
\end{equation} whenever $s, t$ are in $J$ with $s\le t.$
\end{thm} \begin{proof} We introduce the two quadratic forms of the solution \begin{equation} E: = E(t) = \frac{1}{2}(u'^2(t) +b u^2(t)) \end{equation} and 
\begin{equation} F := E + \frac{c}{2}uu' +  \frac{c^2}{4}u^2\end{equation} The quadratic form $F$ is equivalent to the basic form $u'^2+u^2$ since we also have 
$$ F =  \frac{1}{4}u'^2 + \frac{b}{2}u^2 + \frac{1}{4} (u'+ cu)^2 . $$ Now an immediate calculation shows that along any solution curve $u= u(t)$ we have 
$$ F'(t) = -auu'- \frac{c}{2}u'^2 - \frac{cb}{2}u^2- \frac{ca}{2}u^2 $$   We observe that $$ -auu'-  \frac{ca}{2}u^2  =  - \frac{a}{2c} ( u'+ c u ) ^2 + \frac{a}{2c}u'^2$$ so that 
$$ F'(t) \le - (\frac{c}{2}- \frac{a}{2c})u'^2 - \frac{cb}{2}u^2: = - \Phi (u, u' ) $$ 
Now since $0\le a\le C$, the quadratic form $ \Phi (u, v)$  is uniformly (with respect to t) definite positive as soon as $$ \frac{c}{2}>\frac{C}{2c} $$  proving the claim when $C<c^2 .$\\

To conclude when $ C< 2c\sqrt b$ , we rewrite the equation in the form

 \begin{equation}  u''+ cu' + (b+\frac{C}{2})u  + \alpha (t) u= 0\end{equation} where $\displaystyle \alpha(t) = a(t) - \frac{C}{2}$ satisfies $\displaystyle |\alpha(t)| \le \frac{C}{2}$, so that the variable coefficient of $u$ is centred while the fixed coefficient increases by $\displaystyle\frac{C}{2}$.  We consider now the alternative energy function \begin{equation} G := F+ \frac{C}{2} u^2=  \frac{1}{2}(u'^2 +(b+\frac{C}{2}) u^2+ \frac{c}{2}uu' +  \frac{c^2}{4}u^2\end{equation} and we obtain $$ G'(t) = -\alpha uu'- \frac{c}{2}u'^2 - \frac{cb_1}{2}u^2- \frac{c\alpha}{2}u^2 $$ with $\displaystyle b_1= b+\frac{C}{2}.$ Since $\displaystyle\alpha \ge -\frac{C}{2}$ we infer 
 $$ G'(t) \le  -\alpha uu'- \frac{c}{2}u'^2 - \frac{cb}{2}u^2  = -\Psi (u, u') $$ But since $ \displaystyle |\alpha(t)| \le \frac{C}{2}$, a sufficient condition for $ \Psi $ to be uniformly (with respect to t) definite positive is now 
 $$ \left(\frac{C}{2}\right)^2 < 4 \frac{c}{2}.\frac{cb}{2}$$ which reduces to $ C^2< 4c^2b $ as claimed. \end{proof}

\subsection{Comparison with the stability conditions from \cite{F-H2}}\label{sec:Comp}

 In \cite{F-H2}, the authors obtained the exponential stability of the null solution under the stronger assumption $C< c\sqrt{b}$, and a strong stability result for bounded solutions under the assumption $$C\le {c^2\over4}+ c\sqrt{b+{c^2\over 16}},$$ leaving open the question of the rate of convergence to $0$ in the second case. Actually it is easily checked that $$ {c^2\over4}+ c\sqrt{b+{c^2\over 16}} < c\max\{c, 2\sqrt b \}$$ for all positive constants $b, c$, for instance we have the following chain of inequalities 
$$ \sqrt{b+{c^2\over 16}}  < \sqrt{b}+{c\over 4} \le  \max\{{3c\over 4}, 2\sqrt b- {c\over 4} \} = \max\{c, 2\sqrt b \}- {c\over 4}.$$ 
Therefore Theorem \ref{mainth} improves both results from \cite{F-H2} and solves the question on the decay rate. The stability result for equation \eqref{Hilldiss} under the  assumption $C< c\sqrt{b}$ was used in \cite{FilG ?}, and it is reasonable to expect that theorem \ref{mainth} will allow some progress in the study of stability of bridges.  It has been known for a long time that condition \eqref{C} with $C$ arbitrary is not sufficient for the conclusion of the Theorem, because, for instance, Duffing's equation can have multiple periodic solutions. The question of the maximal value of $C$ preserving the result of Theorem \ref{mainth} seems to be quite delicate and we give a partial answer below. 
 
 \subsection{A step towards optimality.}\label{sec:Opt}
 
 The next result shows that in the statement of Theorem \ref{mainth}, the condition \eqref{C2} is sharp up to a factor at most $\frac{\pi}{2}$.
 
 \begin{thm}\label {Countex} For any $\varepsilon>0$, there exists $b>1$ and $c< 2\sqrt b$ such that equation \eqref{Hilldiss} has an unbounded solution on $J = \re^+$ whereas \eqref {C}  holds true with $ C < (\pi + \varepsilon)c \sqrt b.$ \end{thm} \begin{proof} For any number $\omega>1$ we set 

\begin{equation}
\alpha(t)=\left\{
\begin{array}{cc}
\omega^2& \mbox{if }t\in (0, \frac{\pi}{2\omega}), \\[0.5ex]
1& \mbox{if }t\in (\frac{\pi}{2\omega}, \frac{\pi}{2\omega}+ \frac{\pi}{2}),
\end{array}
\right. \end{equation} and 

\begin{equation}
v(t) = \left\{
\begin{array}{cc}
\cos \omega t  & \mbox{if }t\in [0, \frac{\pi}{2\omega}], \\[0.5ex]
-\omega \sin( t - \frac{\pi}{2\omega}) & \mbox{if }t\in [\frac{\pi}{2\omega}, \frac{\pi}{2\omega}+ \frac{\pi}{2}]\end{array}
\right.\end{equation} It is not difficult to check that $v\in C^1 ([0, \frac{\pi}{2\omega}+ \frac{\pi}{2}]$ with \begin{equation}
v'(t) = \left\{
\begin{array}{cc}
-\omega \sin \omega t  & \mbox{if }t\in [0, \frac{\pi}{2\omega}], \\[0.5ex]
-\omega cos( t - \frac{\pi}{2\omega}) & \mbox{if }t\in [\frac{\pi}{2\omega}, \frac{\pi}{2\omega}+ \frac{\pi}{2}]\end{array}
\right.\end{equation} Moreover $v'$ is Lipschiz continuous and $v$ is a solution of \begin{equation} \label{Hill} v''+ \alpha(t) v= 0 \end{equation} Then we can extend $v$ on the full halfline by setting, for each $k\in \mathbb{N}$, 

\begin{equation}
v(t + k (\frac{\pi}{2\omega}+ \frac{\pi}{2}) = \left\{
\begin{array}{cc}
(-\omega) ^k \cos \omega t  & \mbox{if }t\in [0, \frac{\pi}{2\omega}], \\[0.5ex]
(-\omega) ^{k+1} \sin( t - \frac{\pi}{2\omega}) & \mbox{if }t\in [\frac{\pi}{2\omega}, \frac{\pi}{2\omega}+ \frac{\pi}{2}]\end{array} 
\right.\end{equation}and it turns out that $v$ is a solution of \eqref {Hill} on $\re^+$ with $\alpha$  extended on $\re$ by $ (\frac{\pi}{2\omega}+ \frac{\pi}{2})-$ periodicity. Now $v$ is clearly unbounded and if we introduce 
$$ u(t) = \exp(-\frac{c}{2}t) v(t), $$ u will still be unbounded for $c$ small enough. In addition $u$ is a solution on $\re^+$ of the equation 
\begin{equation} \label{Hilldis} u''+ (\alpha(t) + \frac{c^2}{4}) u + cu' = 0 \end{equation} We are now in the situation of Theorem \ref{mainth} with 
$$ b = 1+ \frac{c^2}{4}, \quad a(t) = \alpha(t) -1,\quad C = \omega^2 -1 $$  In order for $u$ to be unbounded, the condition is 
$$ \omega > \exp( \frac{c}{2} (\frac{\pi}{2\omega}+ \frac{\pi}{2}) $$ which reduces to $$ c< \frac{4\omega \ln \omega} {\pi (1+\omega) }:= c_0$$ Of special interest is the case $\omega : = 1+h$ with $h$ tending to zero. In that situation, $c_0$ is equivalent to $\frac{2}{\pi} h$ and $C = \omega^2 -1 $ is equivalent to $2h$, so that 
$$ \frac{C}{c_0\sqrt b} \sim \frac{2h}{\frac{2h}{\pi}} = \pi $$ \end{proof}
\begin{rmk} For the time being we do not have a similar example for large dissipations (i.e. when $c> 2\sqrt b$). On the other hand, since when $a$ is constant there is no condition on its size for exponential stability, it is rather peculiar that we are able to approach the sufficient condition so closely with a ``bang-bang" periodic coefficient $a$ taking only two values $0$ and $C$, moreover with C as small as we wish. \end{rmk}
\begin{rmk} The situation described in our example is a special case of the general phenomenon  called ``parametric resonance" (cf. e.g. \cite{Landau}). But here the calculations are exceedingly simple as well as the geometric ideas behind the formulas. Generally the examples given in the litterature are not computable by elementary functions, here we ``sacrificed" analyticity in order to obtain computable solutions. It is clear that analytic examples with the same properties must also exist. \end{rmk}

\medskip  

\setcounter{equation}{0}\section{Application to dissipative Duffing's equations}\label{sec:Duffing}

Let $g\in C^1(\re)$ be such that  $g(0) = 0$ and for some $b>0$ $$ \forall s\in {\re},\quad g'(s)
\ge  b $$
The following result, improving Theorem 3.1 from \cite{F-H2}, is an immediate consequence of theorem \ref{mainth}  applied to the difference of two solutions :

\medskip \begin{thm} Let $u, v$ two solutions of \eqref{Duff} on $J = [t_0, +\infty)$ and let us set

$$ M = \max \{\L_{ \,t \rightarrow +\infty }\vert u(t)\vert,
\L_{ \,t \rightarrow +\infty }\vert v(t)\vert\} $$
$$ A  = \sup _{s \in [-M, M]} \{ g'(s) -b  \}$$
Then assuming \begin{equation}A< c\max\{c, 2\sqrt b \}. \end{equation} there are $\delta>0$ and $K>0$ such that 
 \begin{equation} \forall t\in J,\quad \vert u(t)-v(t)\vert+\vert u'(t)-v'(t)\vert \le K \exp(-\delta t )
 \end{equation} \end{thm} 
 \medskip  \begin{cor} In the typical case  $$ g(u) = bu + a\vert u\vert^p u
$$
the convergence result is obtained as soon as
$$ \L_{ \,t \rightarrow +\infty }\vert f(t)\vert<\min \{{{bc^{2\over p}} \over {(a(p+1))^{1\over p}}}, {c^{{p+1}\over p}b^{{p+1}\over 2p}\over {2^{{p-1}\over p}(a(p+1))^{1\over p}}}\}$$\end{cor} 
\begin{proof}We use the notation
$$ F =  \L_{ \,t \rightarrow +\infty }\vert f(t)\vert $$ and we employ \eqref{CHbound}. If $c \ge 2\sqrt b$ we have $$ A =  (p+1)a \,\sup \{\vert s\vert^p,\quad 0\le s\le{1\over {b}}F  \} =
 (p+1)a{1\over {b^p}}F^p$$  so that (3.3) reduces to
$$(p+1)a{1\over {b^p}}F^p < c^2$$ or equivalently
$$F < {{bc^{2\over p}} \over {(a(p+1))^{1\over p}}}$$ If $ c \le2\sqrt b$ we have $$ A =  (p+1)a \,\sup \{\vert s\vert^p,\quad 0\le s\le{2\over {c\sqrt b}}F  \} =
 (p+1)a{2^p\over { c^p b^{p/2}}}F^p$$ so that (3.3) reduces to$$ (p+1)a{1\over {c^p b^{p/2}}} 2^p F^p < 2c\sqrt b$$ or equivalently$$F < {c^{{p+1}\over p}b^{{p+1}\over 2p}\over {2^{{p-1}\over p}(a(p+1))^{1\over p}}}$$ \end{proof}
 \begin{rmk}  The extent of sharpness of the new results on Duffing's equation remains basically unknown in both general and special cases, since the optimality of the result on the linear problem does not give any information about the non-linear problem: optimality of a method does not imply optimality of the result. As a consequence, it would be important to quantify as much as possible the non-uniqueness results of the litterature. \end{rmk}
 
\setcounter{equation}{0}\section{Some evolution equations}\label{sec:Evolution}

Let $H$ be a  real Hilbert space.  For every $x$ and $y$ in
$H$, $|x|$ denotes the norm of $x$, and $\langle x,y\rangle$ denotes
the scalar product of $x$ and $y$.  Let $B$ be a self-adjoint linear
operator on $H$ with dense domain $D(B)$.  We assume that $B$ is
coercive, and we consider the second order linear evolution equation
\begin{equation}\label{evol}
	u''(t)+[B + A(t)] u(t)+\gamma u'(t)=0,
	\quad\quad
	t\geq 0,
	\end{equation}
with 
initial data
\begin{equation}
	u(0)=u_{0} \in D(B^{1/2}):= V  ,
	\hspace{3em}
	u'(0)=u_{1} \in H .
	\label{IV}
\end{equation}  
Assuming that $\gamma \in L(H)$ and $A(t) \in L^\infty (\re^+, L(H))$, the initial value problem \eqref {evol}-\eqref {IV} has a mild solution 
$$ u \in C(\re^+, V)\cup C^1(\re^+, H)\cup W^{2,\infty}_{loc} (\re^+, V')$$ which satisfies the classical energy identity. We identify the unbounded operator $B$ with its extension in $L(V, V')$ and we are interested, in view of  applications to Duffing-like infinite dimensional problems, in a criterion for exponential stability of (0,0) in the energy norm. 
We consider first a coercivity constant $b$ for $B$ and a coercivity constant $c$ for $\gamma$
\begin{eqnarray}\label{hypo1}
\exists b>0, \quad  \forall v\in V,\quad \langle Bv,v\rangle\geq b\vert
v\vert^2
\end{eqnarray}
\begin{eqnarray}\label{hypo2}
\  \forall v\in V,\quad \langle \gamma v,v\rangle\geq c|v|^2
\end{eqnarray}

and we introduce the following numbers 
$$ ||A||: = ess \sup_{t\ge 0}  ||A(t) ||_{L(H) } $$  
$$ \Gamma : = ||\gamma||_{L(H)} $$ Note that in general $ c\le \Gamma $ and in the important special case $\gamma = cI$ we have $c= \Gamma $. 
 In the sequel we shall denote  the norm in $V$ by $$\  \forall v\in V,\quad  ||v|| := |B^{1/2} v| $$ and the duality pairing in $V'\times V$ will be denoted in the same way as the inner product in $H$. We now state the main result of this section

\begin{thm}\label {evolth} Let us assume that the operator $\gamma$  is such that $ \gamma \in L(V) $ and for some positive constants $\rho, \eta$ we have 

\begin{eqnarray}\label{hypo2}
\  \forall v\in V,\quad \langle Bv, \gamma v\rangle\geq \rho |v|^2 + \eta ||v||_V ^2
\end{eqnarray}

Assume \begin{equation}\label{A} ||A|| < \sqrt {c\rho + \frac{c^2}{4} \Gamma^2} - \frac{c}{2} \Gamma. \end{equation} Then there are  $\delta>0$ and $M>0$ such that any solution of \eqref {evol}-\eqref {IV}satisfies  \begin{equation}||u(t)||^2 + |u'(t)|^2 \le M [||u(s)||^2 + |u'(s)|^2] \exp(-\delta (t-s))
\end{equation} whenever  $0\le s\le t.$
\end{thm} 
\begin{proof} We introduce the energy  of the solution \begin{equation} E: = E(t) = \frac{1}{2}(|u'(t)|^2  +|B^{1/2}u(t)|^2) \end{equation} and the quadratic form \begin{equation} \Phi =  \frac{1}{2}(|u'(t)|^2  +|B^{1/2}u(t)|^2) +  \frac{1}{2}\langle\gamma u, u' \rangle+  \frac{1}{4}|\gamma u|^2\end{equation}  The form $\Phi$ is definite on $V\times H$ as a function of $(u, u')$ since 
 \begin{equation}  \Phi =  \frac{1}{4}|u'(t)|^2  +\frac{1}{2}|B^{1/2}u(t)|^2 +  \frac{1}{4}|u'+ \gamma u|^2 \end{equation}
 A straightforward calculation gives $$ \Phi'(t) = -\langle A(t) u, u'\rangle -  \frac{1}{2}\langle \gamma u', u'\rangle - \frac{1}{2} \langle  A(t)u - B u, \gamma u \rangle $$
 Therefore $$ \Phi'(t) \le  - \frac{c}{2}|u'|^2 + ||A|| |u| |u'| +  \frac{1}{2}||A|| \Gamma |u|^2-  \frac{\rho}{2}|u|^2-  \frac{\eta}{2}||u||^2 $$ Leaving aside the last term, the quadratic form  
 $$ \frac{c}{2}|u'|^2 -||A|| |u| |u'| -  \frac{1}{2}||A|| \Gamma |u|^2 +  \frac{\rho}{2}|u|^2 $$  will be definite positive in the sense of $H\times H$ whenever 
 $$ ||A||^2 < c (\rho - ||A|| \Gamma) \Longleftrightarrow (||A|| + \frac{c}{2}\Gamma)^2 < c \rho + \frac{c^2}{4}\Gamma^2. $$
 The conclusion then follows easily by using the last term which provides definiteness  in the sense of $V\times H$  . \end{proof}
 
 \begin{rmk}  For applications it may be interesting to consider non-constant dissipations, such as $ \gamma u (x) = c(x) u(x) $  when $H$ is a Hilbert space of functions. When the dissipation is constant, the formulas simplify. \end{rmk}
 
 \begin{cor}\label {cI} Let us assume that  $\gamma = cI $ 
Assume \begin{equation}\label{A} ||A|| < c  \sqrt {b + \frac{c^2}{4} } - \frac{c^2}{2} . \end{equation} Then there are  $\delta>0$ and $M>0$ such that any solution of \eqref {evol}-\eqref {IV}satisfies  \begin{equation}||u(t)||^2 + |u'(t)|^2 \le M [||u(s)||^2 + |u'(s)|^2] \exp(-\delta (t-s))
\end{equation} whenever  $0\le s\le t.$
\end{cor} 

 \begin{rmk}  In case $H= \re$ and $Bu = (b+ \frac{C}{2}) u $ with $ ||A|| =  \frac{C}{2}$, the condition reduces to $C < 2c\sqrt b$ and we recover the result of Theorem \ref{mainth}. \end{rmk}

  \begin{rmk}  In the applications to infinite dimensional stability results, as in the case of Duffing's equation we need a priori information on the ultimate bound of solutions to be able to use the previous theorems. Some useful results in this direction are proved in  \cite{F-H1}, cf. also \cite{AH} for sharp results in the linear case allowing more precise ideas  on optimality
  of such estimates.\end{rmk}
  
   \begin{rmk}  It is clear from the proof that some stability result will also be available when the condition $A(t) \in L^\infty (\re^+, L(H))$ is replaced by $A(t) \in L^\infty (\re^+, L(V, H)).$  But of course the conditions will be more complicated to write down and a generalization in this direction must be motivated by concrete applications. \end{rmk}

\setcounter{equation}{0}\section{A simple example }\label{sec:Example} To illustrate the last theorem, we give a simple explicit example in one space dimension. Let us consider the semilinear wave equation 
 \begin{equation}\label{wave}u_{tt} - u_{xx} + ku^3 + c u_t = f(t, x) \end{equation} in the interval $\Omega = (0, 1)$  with homogeneous Dirichlet boundary conditions 
 $$ u(t, 0) = u(t, 1) = 0$$ Assuming $f\in L^\infty ( \re^+; L^2(\Omega)) $, and denoting by $|.|_H$ the norm in $H = L^2(\Omega)$,  it follows from theorem 2.1 of  \cite{F-H1} that all solutions are bounded in the energy space with 
 \begin{equation} \displaystyle\limsup_{ \,t \rightarrow +\infty }\vert u_x(t)\vert^2_H \le \left(\frac{1}{\pi^2} + \frac{4}{c^2}\right)  \displaystyle\limsup_{ \,t \rightarrow +\infty }\vert f(t)\vert^2_H  \end{equation} If $u, v$ are two different solutions of the equation, the difference $w = u-v$ satisfies the equation 
 \begin{equation}\label{A}w_{tt} - w_{xx} + k(u^2+uv+v^2)w + c w_t = 0 \end{equation} and it is rather classical to observe that in this case 
 $$ ||(u^2+uv+v^2)||_\infty  \le 3 \max\{||u||^2_\infty , ||v||^2_\infty\} \le \frac{3}{\pi} \max\{||u_x||^2_H , ||v_x||^2_H\}$$ As a consequence, by Corollary \ref{cI}, two arbitrary solutions of \eqref{wave} will be exponentially asymptotic in the energy space as soon as we have 
 $$ k \frac{3}{\pi} \left(\frac{1}{\pi^2} + \frac{4}{c^2}\right)  \displaystyle\limsup_{ \,t \rightarrow +\infty }\vert f(t)\vert^2_H <  c  \sqrt {\pi^2 + \frac{c^2}{4} } - \frac{c^2}{2}$$ 
 For instance, the special case $c= 2\pi$ gives the sufficient condition $$ k  \displaystyle\limsup_{ \,t \rightarrow +\infty }\vert f(t)\vert^2_H <  \frac{\pi^5}{3} (\sqrt 2-1)$$ Even for $k = 1$ this allows rather large source terms.

\label{NumeroPagine}

\end{document}